\documentclass[a4paper,12pt]{amsart}
\usepackage{mathrsfs}
\usepackage{amsfonts}
\usepackage{amssymb}
\usepackage{amsmath}
\usepackage{amsthm}
\usepackage{booktabs,array}
\usepackage{color}

\newcommand{\ID}{{\mathbb D}}

\newtheorem{theorem}{Theorem}[section]

\newtheorem{lemma}[theorem]{Lemma}
\newtheorem{corollary}[theorem]{Corollary}

\newtheorem{remark}[theorem]{Remark}

\newtheorem{prob}{Problem}
\newcommand{\bprob}{\begin{prob}}
\newcommand{\eprob}{\end{prob}}

\newcommand{\IN}{{\mathbb N}}

\makeatletter
\@namedef{subjclassname@2020}{%
  \textup{2020} Mathematics Subject Classification}
\makeatother

\def\be{\begin{equation}}
\def\ee{\end{equation}}

\begin{document}

\bibliographystyle{amsplain}

\title[Polyanalytic and Log-$\alpha$-analytic functions]
{Landau-type Theorems for Polyanalytic and Log-$\alpha$-analytic functions}

\author{Ping Li}
	\address{P. Li, School of Mathematical Sciences, South China Normal University, Guangzhou, Guangdong 510631, China.}
	\email{lp0150321@163.com}
	
\author{Ming-Sheng Liu${}^{~\mathbf{*}}$
	}
\address{M.S. Liu, School of Mathematical Sciences, South China Normal University, Guangzhou, Guangdong 510631, China.} \email{liumsh65@163.com}

\author{Saminathan Ponnusamy
}
\address{S. Ponnusamy, Department of Mathematics,
Indian Institute of Technology Madras, Chennai-600 036, India. }


\email{samy@iitm.ac.in}

\author{Hanghang Zhao}
\address{H. Zhao, School of Mathematical Sciences, South China Normal University, Guangzhou, Guangdong 510631, China.}
\email{zhh18888106641@163.com}

%
%
%
%

\date{\today}

\keywords{Polyanalytic functions, Log-$\alpha$-analytic functions, Coefficient estimates, Landau-type theorems\\
{\tt $^*$Corresponding author.}
}


\begin{abstract}
In the present article, we investigate  the univalence property of polyanalytic functions and $\log$-$\alpha$-analytic functions. First,   by using a new idea, we prove an improved lemma and the coefficient estimates for bounded polyanalytic functions on the unit disk. Then, we present three versions of Landau-type theorems for such functions and determine the univalence domain and the radius of schlicht disk. Finally, as a consequence, the Landau-type theorems for $\log$-$\alpha$-analytic functions are also provided.
\end{abstract}

\maketitle


\section{Introduction and Preliminaries}\label{HLP-sec1} 

Landau-Bloch type theorems for different families of (holomorphic, harmonic, biharmonic, pluriharmonic) functions are obtained by a number of researchers (cf. \cite{CRW2014, CPR2014, CZ2019,G2006,LP-2022,LXY2017} and the references therein).
On most cases, the inequalites/results obtained in this direction are not guaranteed to be sharp. Therefore, one requires new methods to establish non-trivial and improved estimates if not sharp. It it worth recalling that the theory of polyanalytic functions is an interesting topic in the sense that it naturally extends the concept of holomorphic functions to null-solutions of higher-order powers of the Cauchy-Riemann operator. These functions were introduced in 1908 by Kolosso \cite{Kgv2006} to study elasticity problems. For a complete introduction to nonanalytic functions and their basic properties we refer to \cite{Bm1991,Bm1997}. The class of nonanalytic functions have been studied by various authors from different perspectives; see\cite{Ald2010,Aml2011} and the references therein.

In this paper we consider non-analytic functions which do have common properties similar to those of the case of analytic functions, and they have a notable structure and applications \cite{Ald2010,Bm1991,Bm1997}. For example, Landau-type theorems play pivotal roles in the analytic function theory, which provide the core theoretical basis for the study of univalence and image domain properties of different classes of functions
(cf. \cite{CGH2000,LP-2022} and the references therein). With the development of polyanalytic functions, it becomes a natural and important research direction to extend, for example, the Landau-type theorem to the case of polyanalytic functions.

\subsection{Definitions and Notation} 
 A continuous complex-valued function~$F$ defined in a domain $\Omega\subseteq \mathbb{C}$ is polyanalytic of order $\alpha\in \IN :=\{1,2, \ldots\}$ if it satisfies the generalized Cauchy-Riemann equations
 $$
\frac{\partial^\alpha}{\partial \bar{z}^\alpha} F(z)=0, \text { for all } z \in \Omega .
$$
Every polyanalytic function $F$ of order $\alpha$ can be decomposed in terms of  $\alpha$ analytic functions of the form
$$
F(z)=\sum_{k=0}^{\alpha-1} \bar{z}^k A_k(z),
$$
where each $A_k$ is analytic for $k \in\{0,1, \ldots, \alpha-1\}$.
Now, we recall some standard notation: Let $\ID$ denote the open unit disk in the complex plane $\mathbb{C}$.
For $r>0$, we let $\mathbb{D}_{r}=\{z\in\mathbb{C}:\,|z|<r\}$ so that  $\mathbb{D}:=\mathbb{D}_{1}$. For a continuously differentiable complex-valued mapping $F=u+iv$, that is to say that $F$ is a $C^1$-function, on a domain $\Omega$, the formal partial derivatives are defined by
$$
F_{z}=\frac{1}{2}(F_{x}-iF_{y})~\mbox{ and }~ F_{\overline{z}}=\frac{1}{2}(F_{x}+iF_{y}),
$$
where $z=x+iy \in\Omega$. Also, the maximal and minimal stretching of a $C^1$ mapping $F$ 
are defined respectively by
\begin{eqnarray*}
\Lambda_F (z) &=& \max_{0\leq t\leq2\pi}|F_z(z)+e^{-2it}F_{\overline{z}}(z)|
=|F_z(z)|+|F_{\overline{z}}(z)|,
\end{eqnarray*}

and
\begin{eqnarray*}
\lambda_F (z)&=& \min_{0\leq t\leq2\pi}|F_z(z)+e^{-2it}F_{\overline{z}}(z)|=\big |\,|F_z(z)|-|F_{\overline{z}}(z)| \big |.
\label{1-2}
\end{eqnarray*}
Note that the Jacobian of $F$ is given by $J_F=|F_z|^2-|F_{\overline{z}}|^2$ (see \cite{CGH2000}).

A function $f$ is called a $\log$-$\alpha$-analytic in a domain $D \subset \mathbb{C}$ if $\log f$ is a polyanalytic function of order $\alpha$ in $D$. Here and in what follows, ``$\log$" is taken to be its principal branch such that $\log 1=0$. When $\alpha=1$, the function $f$ is called $\log$-analytic. Then it follows from (1.1) that $f$ is $\log$-$\alpha$-analytic in a simply connected domain $D$ if and only if $f$ can be written as $$
f(z)=\prod_{k=0}^{\alpha-1}\left(a_k(z)\right)^{\bar{z}^k},
$$
where each $a_k(z)$ is $\log$-analytic in $D$ for $k \in\{0,1, \ldots, \alpha-1\}$ (see \cite{ZLK2025}).




\subsection{Landau-type theorems}

Let
\begin{eqnarray*}
	\mathcal{A}&=&\left\{f: f \text{ is analytic in } \mathbb{D}, f(0)=0 \text{ and }f^\prime(0)=1\right\},~\mbox{ and }\\
 	\mathcal{A}(M)&=&\left\{f\in \mathcal{A}: |f(z)|\leq M \text{ in }\mathbb{D}\right\}.
\end{eqnarray*}
A well-known classical theorem of Landau  asserts that if $f\in \mathcal{A}(M)$ for some $M>1$, then $f$ is univalent in $\mathbb{D}_{r_0}$ and $f\left(\mathbb{D}_{r_0}\right)$ contains a disc $\mathbb{D}_{\sigma_0}$, where $r_0=M-\sqrt{M^2-1}$ and $\sigma_0=M r_0^2$. The quantities $r_0$ and $\sigma_0$ are best possible as the function $f_0(z)=M z\left(\frac{1-M z}{M-z}\right)$ shows (see \cite{L1926}).
The Bloch theorem asserts the existence of a positive constant $b$ such that if $f$ is an analytic function on $\mathbb{D}$ with $f'(0) = 1$, then $f(\mathbb{D})$ contains a schlicht disk of radius $b$, that is, a disk of radius $b$ which is the univalent image of some region in $\mathbb{D}$. The supremum of all such constants $b$ is called the Bloch constant (see \cite{{CGH2000}}).

In 2000, Chen et al. \cite{CGH2000} obtained two non-sharp versions of the Landau theorem for bounded harmonic mappings of the unit disk. Then after many authors investigated the Landau-type theorems for harmonic mappings and improved their results. See for example, Chen et al. \cite{CGH2000,CPR2014}, Grigoryan \cite{G2006} and Liu \cite{LC2018,LLL2020}. Using the work of Abdulhadi and Abu Muhanna \cite{AM2008}, several authors
investigated the Landau-Bloch type theorems for polyharmonic mappings with a different normalization (see  \cite{BL2019,CRW2014,LXY2017}).

In 2022, the Landau-type theorem for polyanalytic functions was first studied by Abdulhadi and Hajj \cite{AH2022}.

\begin{theorem}{\rm \cite{AH2022}}
Let $F(z)=\sum_{k=0}^{\alpha-1} \bar{z}^k A_k(z)$ be a polyanalytic function of order $\alpha$ on $\mathbb{D}$ with $\alpha \geq 2$, where $A_k\in \mathcal{A}(M)$ for each $k\in \{0,\ldots , \alpha-1\}$, and $M>1$. Then there is a constant $0<\rho_1<1$ so that $F$ is univalent in $|z|<\rho_1$. In particular $\rho_1$ satisfies $$
1-M\left(\frac{\rho_1\left(2-\rho_1\right)}{\left(1-\rho_1\right)^2}+\sum_{k=1}^{\alpha-1} \frac{\rho_1^k\left(1+k-k \rho_1\right)}{\left(1-k \rho_1\right)^2}\right)=0,
$$
and $F\left(\mathbb{D}_{\rho_1}\right)$ contains the disk $\mathbb{D}_{R_1}$, where
$$
R_1=\rho_1-\rho_1^2\left(\frac{1-\rho_1^{\alpha-1}}{1-\rho_1}\right)-M \sum_{k=0}^{\alpha-1} \frac{\rho_1^{k+2}}{1-\rho_1}.
$$
\end{theorem}

In \cite{CRW2014}, under a suitable restriction, Chen et al. established the following Landau-type theorem for bounded polyharmonic mappings.

\begin{theorem}\label{Theo-A} {\rm (\cite[Theorems 1 and 2]{CRW2014})} Suppose that $F$ is a polyharmonic mapping of the form:
$$
F(z)=a_0+\sum^p_{k=1}|z|^{2(k-1)}\sum^{\infty}_{n=1}\left(a_{n,k}z^n+\overline{b_{n,k}z^n}\right),
$$
and all its non-zero coefficients $a_{n,k_1}, a_{n,k_2}$ and $b_{n,k_3}, b_{n,k_4}$ satisfy the condition:
$$
\left|\arg\frac{a_{n,k_1}}{a_{n,k_2}}\right|\le \frac{\pi}{2} \quad and \quad\left|\arg\frac{b_{n,k_3}}{b_{n,k_4}}\right|\le \frac{\pi}{2}.
$$

If $|F(z)| \leq M$ in $\mathbb{D}$ for some $M>1$ and $F(0)=0=J_F(0)-1 $, then 
$F$ is univalent in the disk $\mathbb{D}_{\rho_2}$ and $F(\mathbb{D}_{\rho_2})$ contains a schlicht disk $\mathbb{D}_{R_2}$,
where $\rho_2=\rho_2(M,p)$ is the least positive root of the following equation:
$$
1-\sqrt{M^4-1}\left(\frac{2r-r^2}{(1-r)^2}+\sum^{p-1}_{k=1}\frac{r^{2k}}{(1-r)^2}
+\sum^{p-1}_{k=1}\frac{2kr^{2k}}{1-r}\right)=0,
$$
and
$$
R_2=\lambda_0'(M)\rho_2\left(1-\sqrt{M^4-1}\frac{\rho_2}{1-\rho_2}-\sqrt{M^4-1}\sum^{p-1}_{k=1}\frac{2\rho_2^{2k}}{1-\rho_2}\right),
$$
with
\begin{eqnarray*}
\lambda_0'(M)=\left\{
       \begin{array}{ll}
         \frac{\sqrt{2}}{\sqrt{M^{2}-1}+\sqrt{M^{2}+1}}, & {1\leq M\leq M_{0}=\frac{\pi}{2\sqrt[4]{2\pi^2-16}}\approx1.1296,} \\
         \frac{\pi}{4M}, & {M>M_{0}}.
       \end{array}
     \right.
\end{eqnarray*}

In particular, if $|F(z)| \leq 1$ for all $z\in\mathbb{D}$ and $F(0)=0=J_F(0)-1 $, then $F$ is a univalent mapping of $\mathbb{D}$ onto $\mathbb{D}$.
\end{theorem}


Now the results due to Abdulhadi et al. \cite{AH2022} and Chen et al. \cite{CGH2000} motivate us to raise the following:

\bprob\label{HLP-prob1}
Can we establish the Landau-type theorems for bounded polyanalytic functions?
\eprob

The objective of this paper is to provide an affirmative answer to Problem \ref{HLP-prob1} (see Theorems \ref{thm3}-\ref{thm5}).

The paper is organized as follows.  In Section \ref{HLP-sec2}, we first establish the coefficients estimates for the bounded polyanalytic functions in $\ID$. In Section \ref{HLP-sec3}, we present different kinds of Landau-type theorems for the bounded polyanalytic functions in the unit disk.
In Section \ref{HLP-sec4}, as a consequence of our investigations, we present three versions of Landau-type theorems for certain $\log$-$\alpha$-analytic functions in $\ID$.

\section{Key Lemma, theorem and their Proofs} 
\label{HLP-sec2}
We begin to establish the following:

\begin{lemma}\label{lem0}
Suppose that $a>0$, $f(x)$ is continuous on $[0, a]$ and differentiable in $(0, a)$ and $\frac{f'(x)}{x}\leq f(x)$ in $(0, a)$. Then the function $e^{-\frac{x^2}{2}}f(x)$ is decreasing in $[0, a]$.
\end{lemma}

\begin{proof}
Let $g(x)=e^{-\frac{x^2}{2}}f(x)$ for $x\in [0, a]$. Then
\begin{eqnarray*}
g'(x)=(f'(x)-x f(x))e^{-\frac{x^2}{2}}\leq 0 ~\mbox{ for $x\in (0, a)$,}~
\end{eqnarray*}
which implies that the function $g(x)=e^{-\frac{x^2}{2}}f(x)$ is decreasing in $(0, a)$. Finally, the
continuity of $f$ on $[0,a]$ concludes the proof.
\end{proof}

Next, we recall a lemma from \cite{ZLK2025} and extend their condition from $0<\sigma< 1$ to $0<\sigma\leq 1$, and provide a new proof based on Lemma \ref{lem0}, which plays a key role in the proof of our main results.

\begin{lemma}{\rm \cite{ZLK2025}}\label{lem1}
Suppose that $\alpha \in \IN$, $0<\sigma\leq 1$, and  $0<\rho \leq 1$. Let $f$ be a $\log$-$\alpha$-analytic function satisfying $f(0)=\lambda_f(0)=1$, and  that $f$ is univalent in $\mathbb{D}_\rho$ and $F\left(\mathbb{D}_\rho\right) \supset \mathbb{D}_\sigma$, where $F(z)=\log f(z)$. Then the range $f\left(\mathbb{D}_\rho\right)$ contains a schlicht disk
$\mathbb{D}\left(w_0, r_0\right)=\left\{w \in \mathbb{C}:\, | w-w_0 \mid<r_0\right\}$, where
$$
w_0=\cosh \sigma, \quad r_0=\sinh \sigma .
$$
Moreover, if $\rho$ is the biggest univalent radius of $f$, then the radius $r_0=\sinh\sigma$ is sharp.
\end{lemma}

\begin{proof}
Consider $F(z)=\log f(z)$, \text{ where } $f(z)$ is univalent in $\mathbb{D}_{\rho}$ and $F(\mathbb{D}_{\rho})\supset \mathbb{D}_{\sigma}$. Thus, we have
\begin{eqnarray*}
f(\mathbb{D}_{\rho})\supset e^{\mathbb{D}_{\sigma}}:=\{w=e^z|z\in \mathbb{D}_{\sigma}\}.
\end{eqnarray*}
Now, we wish to prove the following domain inclusion,
\begin{eqnarray}
\mathbb{D}(w_0, r_0)\subset e^{\mathbb{D}_{\sigma}}.
\label{liu28}
\end{eqnarray}
In fact, for $w=u+iv=e^{\sigma (\cos\theta+i\sin\theta)}\in\partial e^{\mathbb{D}_{\sigma}},\, \theta\in\mathbb{R}$, we have
$$
u = e^{\sigma\cos\theta}\cos(\sigma\sin\theta),\quad v = e^{\sigma\cos\theta}\sin(\sigma\sin\theta).
$$
Because $\partial e^{\mathbb{D}_{\sigma}}$ is symmetric with respect to the real axis, we only need to prove that
\begin{eqnarray*}
( u - \cosh\sigma)^2+v^2 = e^{2\sigma\cos\theta}-2\cosh\sigma\cos(\sigma\sin\theta)e^{\sigma\cos\theta}+\cosh^2\sigma\geq\sinh^2\sigma,\quad \theta\in [0, \pi],
\end{eqnarray*}
which is equivalent to
\begin{eqnarray}
 \cosh(\sigma\cos\theta)-\cosh\sigma\cos(\sigma\sin\theta)\geq 0,\quad \theta\in [0, \pi].
\label{liu210}
\end{eqnarray}
Thus, it suffices to prove the inequality (\ref{liu210}). For this purpose, we set $x=\sigma\cos\theta$. Then $\sigma\sin\theta = \sqrt{\sigma^2-x^2}$ for $\theta\in [0, \pi]$. For a fixed $\sigma\in (0, 1]$, set
$$
g_{\sigma}(x)=\cosh x-\cosh\sigma\cos(\sqrt{\sigma^2-x^2}) ~\mbox{ for $ x\in [-\sigma, \sigma].$}
$$
Obviously, $g_{\sigma}(x)$ is continuous on the closed interval $[-\sigma, \sigma]$, and
\begin{eqnarray*}
g_{\sigma}'(x)=\sinh x-\cosh\sigma\cdot\frac{x\sin(\sqrt{\sigma^2-x^2})}{\sqrt{\sigma^2-x^2}}, \quad x\in (-\sigma, \sigma).
\end{eqnarray*}

Note that
\begin{eqnarray*}
\frac{\sin x}{x}>\cos x ~\mbox{ and }~ \frac{\sinh x}{x}<\cosh x ~\mbox{ for $ x\in (0,1)$,}
\end{eqnarray*}
so that
\begin{eqnarray*}
\frac{g_{\sigma}'(x)}{x}&=&\frac{\sinh x}{x}-\cosh\sigma\cdot\frac{\sin(\sqrt{\sigma^2-x^2})}{\sqrt{\sigma^2-x^2}}\\
&<& \cosh  x-\cosh\sigma\cos(\sqrt{\sigma^2-x^2})=g_{\sigma}(x),
\end{eqnarray*}
for $x\in (0,\sigma)$. By Lemma \ref{lem0}, we obtain that $e^{-\frac{x^2}{2}}g_{\sigma}(x)$ is decreasing on $[0, \sigma]$ and thus, we have
\begin{eqnarray*}
e^{-\frac{x^2}{2}}g_{\sigma}(x)\geq e^{-\frac{\sigma^2}{2}}g_{\sigma}(\sigma)=0 ~\mbox{ for $x\in [0, \sigma].$}
\end{eqnarray*}
Note that $g_{\sigma}(x)$ is an even function of $x$ in $[-\sigma, \sigma]$ which implies that $g_{\sigma}(x)\geq 0$ for $x\in [-\sigma, \sigma]$. Thus, the inequality (\ref{liu210}) holds; that is, the inclusion relation (\ref{liu28}) holds and hence, the range $f(\mathbb{D}_{\rho})$ contains a schlicht disk $\mathbb{D}(w_0, r_0)$.

The proof of the sharpness of $r_0=\sinh\sigma$ is the same as that of \cite[Lemma 2.4]{LL2021}.
\end{proof}




\begin{remark}
We note that in the proof of \cite[Lemma 2.4]{LL2021}, the authors used a fact:
For a fixed $\sigma\in (0, 1)$, the function
\begin{eqnarray*}
h_{\sigma}(\theta):=\cosh\sigma\sin(\sigma\sin\theta)-\sinh(\sigma\cos\theta)\tan\theta
\end{eqnarray*}
is increasing for $\theta\in [0, \pi/2)$. In fact, it is non-trivial to prove this fact. The proof of Lemma \ref{lem1} also provides a new and elegant proof of Lemma 2.4 of \cite{LL2021}.
\end{remark}

Now, we establish the coefficient estimates for bounded polyanalytic functions 
in the unit disk, which has an independent interest. 

\begin{theorem}\label{thm1} Suppose that $\alpha \geq 1$ and $F(z)=\sum_{k=0}^{\alpha-1} \bar{z}^k A_k(z)$ is a polyanalytic function of order $\alpha$ on $\mathbb{D}$ and $|F(z)|\leq M$ for $z\in\mathbb{D}$. Assume that for each $k\in\{0,\ldots , \alpha-1\}$, $A_k(z)=\sum_{n=0}^{\infty} a_{n, k} z^n$ is analytic on $\mathbb{D}$, and all its non-zero coefficients $a_{n_1,k_1}, a_{n_2,k_2}$ satisfy
\begin{equation}
\left|\arg \frac{a_{n_1, k_1}}{a_{n_2, k_2}}\right| \leq \frac{\pi}{2}
\label{2-3}
\end{equation}
 for each $k_1,k_2\in\{0,\ldots , \alpha-1\}$ and $n_1, n_2\in\mathbb{N}$ with $k_1\neq k_2,\, n_1\neq n_2$.
\begin{enumerate}
\item[{\rm (1)}] If $A_k\in\mathcal{A}$ for each $k\in \{0,\ldots , \alpha-1\}$, then $M\geq \sqrt{\alpha}~ (\geq 1)$, and
\begin{eqnarray}
\sum_{k=0}^{\alpha-1} \sum_{n=1}^{\infty}\left|a_{n,k}\right|^2 \leqslant M^2.
\label{2-4}
\end{eqnarray}
In particular, we have
\begin{eqnarray}
\left|a_{n, k}\right| \leq \sqrt{M^2-\alpha}  ~\mbox{ for $ n \geq 2$,  and $ k=0,1,\ldots, \alpha-1$}.
\label{2-5}
\end{eqnarray}

\item[{\rm (2)}]
 If $F(0)=\left|J_F(0)\right|-1=0$, then for each $(n, k) \neq(1,0),(0,1),$ we have
\begin{eqnarray}
 \left|a_{n, k}\right| \leq \sqrt{M^2-1},
 \label{2-6}
\end{eqnarray}
 and
\begin{eqnarray}
\lambda_F(0) \geq \lambda_0(M):=\frac{\sqrt{2}}{\sqrt{M^2-1}\,+\sqrt{M^2+1}}.
\label{2-7}
\end{eqnarray}

\item[{\rm (3)}]
If $F(0)=\lambda_F(0)-1=0$, then for each $(n, k) \neq(1,0),(0,1)$, we have
\begin{eqnarray}
 \left|a_{n, k}\right| \leq \sqrt{M^2-1}.
 \label{2-8}
\end{eqnarray}
\end{enumerate}
\end{theorem}

\begin{proof} (1) By assumption $A_k\in \mathcal{A}$, and thus, we have $a_{0, k}=0, a_{1, k}=1$ for each $k\in\{0,\ldots ,\alpha-1\}$.

Note that for $z=r\, e^{i\theta}\, (0<r<1),\, \theta\in [0, 2\pi]$, we have
\begin{eqnarray*}
F(r e^{i\theta})&=&\sum_{k=0}^{\alpha-1} \bar{z}^k A_k(z)=\sum_{k=0}^{\alpha-1}\sum_{n=0}^{\infty} a_{n, k} r^{n+k}\, e^{i(n-k)\theta}\\
&=&\sum_{j=-\alpha+1}^{-1}\Big(\sum_{k=-j}^{\alpha-1}a_{j+k, k} r^{j+2k}\Big) e^{ij\theta}+\sum_{j=0}^{\infty}\Big(\sum_{k=0}^{\alpha-1}a_{j+k, k} r^{j+2k}\Big) e^{ij\theta}.
\end{eqnarray*}

A standard argument using Parseval's identity gives
\begin{eqnarray*}
\frac{1}{2 \pi} \int_0^{2 \pi}\left|F\big (r e^{i \theta}\big)\right|^2 d \theta=\sum_{j=-\alpha+1}^{-1}\Big|\sum_{k=-j}^{\alpha-1}a_{j+k, k} r^{j+2k}\Big|^2 +\sum_{j=0}^{\infty}\Big|\sum_{k=0}^{\alpha-1}a_{j+k, k} r^{j+2k}\Big|^2.
\end{eqnarray*}

In view of the condition (\ref{2-3}), it follows that $\mathrm{Re}(a_{n_1,k_1}\overline{a_{n_2,k_2}})\ge 0$ for each $k_1,k_2\in\{0,\ldots , \alpha-1\}$ and $n_1, n_2\in\mathbb{N}$ with $k_1\neq k_2,\, n_1\neq n_2$, and thus, as $\left|F\left(r e^{i \theta}\right)\right| \leq M$, we obtain that
\begin{equation*}
\sum_{j=-\alpha+1}^{-1}\sum_{k=-j}^{\alpha-1}|a_{j+k, k}|^2 r^{2j+4k} +\sum_{j=0}^{\infty}\sum_{k=0}^{\alpha-1}|a_{j+k, k}|^2 r^{2j+4k}\leq\frac{1}{2 \pi} \int_0^{2 \pi}\left|F\big (r e^{i \theta}\big)\right|^2 d \theta\leq M^2.
\end{equation*}

Letting $r \to 1^{-}$ in the above inequality, note $a_{0, k}=0$ for each $k\in\{0,\ldots ,\alpha-1\}$, we obtain that 
\begin{equation*}
\sum_{j=-\alpha+1}^{-1}\sum_{k=-j}^{\alpha-1}|a_{j+k, k}|^2 +\sum_{j=0}^{\infty}\sum_{k=0}^{\alpha-1}|a_{j+k, k}|^2=\sum_{k=0}^{\alpha-1} \sum_{n=1}^{\infty}\left|a_{n, k}\right|^2\leq M^2.
\end{equation*}
Hence, the inequality (\ref{2-4}) holds.

As $a_{1,k}=1$ for each $k\in\{0,\ldots,\alpha-1\}$, (\ref{2-4}) is equivalent to

$$\alpha+ \sum_{k=0}^{\alpha-1} \sum_{n=2}^{\infty}\left|a_{n, k}\right|^2 \leq M^2,$$

and the inequalities (\ref{2-5}) hold and $M\geq\sqrt{\alpha}\geq 1$.

(2) A direct calculation shows that
\begin{equation*}
F_z(z)=\sum_{k=0}^{\alpha-1} \overline{z}^k A_k^{\prime }(z)=\sum_{k=0}^{\alpha-1} \overline{z}^k \sum_{n=1}^{\infty} n a_{n, k} z^{n-1},
\end{equation*}
and
\begin{equation*}
F_{\bar{z}}(z)=\sum_{k=0}^{\alpha-1} k \overline{z}^{k-1} A_k(z)=\sum_{k=0}^{\alpha-1} k \overline{z}^{k-1} \sum_{n=0}^{\infty} a_{n, k} z^n.
\end{equation*}

Thus, $F_z(0)=A_0^{\prime }(0)=a_{1,0}$ and $F_z(0)=A_1(0)=a_{0,1}$. Since $F(0)=|J_F(0)|-1=0$, we have $a_{0,0}=0$ and $\left|\left|a_{1,0}\right|^2-\left|a_{0,1}\right|^2\right|=1$, it follows from the triangle inequality that
\begin{eqnarray}
\left|a_{1,0}\right|^2+\left|a_{0,1}\right|^2 \geq\left|\left|a_{1,0}\right|^2-\left|a_{0,1}\right|^2\right|=1.
\label{2-11}
\end{eqnarray}

Similar to the proof of the case (1), we get
\begin{eqnarray}
\sum_{k=0}^{\alpha-1} \sum_{n=0}^{\infty}\left|a_{n, k}\right|^2 =\sum_{k=0}^{\alpha-1}\left(\left|a_{0, k}\right|^2+\left|a_{1, k}\right|^2+\sum_{n=2}^{\infty}\left|a_{n, k}\right|^2\right)\leq M^2,
\label{2-12}
\end{eqnarray}
or equivalently,
\begin{eqnarray*}
\begin{aligned} \sum_{k=0}^{\alpha-1} \sum_{n=0}^{\infty}\left|a_{n, k}\right|^2& =\left|a_{0,1}\right|^2+\left|a_{1,0}\right|^2+\sum_{k=2}^{\alpha-1}\left|a_{0, k}\right|^2+\sum_{k=1}^{\alpha-1}\left|a_{1, k}\right|^2+\sum_{k=0}^{\alpha-1} \sum_{n=2}^{\infty}\left|a_{n, k}\right|^2 \leq M^2\end{aligned}
\end{eqnarray*}
(since $a_{0,0}=0$). From (\ref{2-11}) and (\ref{2-12}), we obtain
\begin{eqnarray*}
\sum_{k=2}^{\alpha-1}\left|a_{0, k}\right|^2+\sum_{k=1}^{\alpha-1}\left|a_{1, k}\right|^2+\sum_{k=0}^{\alpha-1} \sum_{n=2}^{\infty}\left|a_{n, k}\right|^2 &\leq M^2-\left(\left|a_{0,1}\right|^2+\left|a_{1,0}\right|^2\right) &\leq M^2-1,
\end{eqnarray*}
and hence, for each $(n, k) \neq(1,0),(0,1)$, we have $ \left|a_{n, k}\right| \leq \sqrt{M^2-1}$ and $M\geq 1$.

If $J_F(0)=1$, then we have
\begin{eqnarray*}
\left|a_{1,0}\right|=\sqrt{1+\left|a_{0,1}\right|^2} \geq 1.
\end{eqnarray*}

Also from (\ref{2-12}), we have
\begin{eqnarray}
\left|a_{0,1}\right|^2+\left|a_{1,0}\right|^2 \leq M^2,
\label{2-14}
\end{eqnarray}
and a comparison with (\ref{2-14}) shows that $M\geq 1$, and
\begin{eqnarray}
\left|a_{0,1}\right| \leq \sqrt{\frac{M^2-1}{2}}.
\label{2-15}
\end{eqnarray}
Thus we find that
$$
\begin{aligned} \lambda_F(0)&=\left|a_{1,0}\right|-\left|a_{0,1}\right|=\sqrt{1+\left|a_{0,1}\right|^2}-\left|a_{0,1}\right|
=\frac{1}{\sqrt{1+\left|a_{0,1}\right|^2}+\left|a_{0,1}\right|}
\\& \geq \frac{\sqrt{2}}{\sqrt{M^2-1}\,+\sqrt{M^2+1}}=:\lambda_0(M).
\end{aligned}
$$

If $J_F(0)=-1$, then by a proof similar as above, we find that 
the inequality (\ref{2-7}) holds.

(3) Since $\big |\vert a_{1,0}\vert-\vert a_{0,1}\vert\big |=\lambda_F(0)=1$, we see that $\vert a_{1,0}\vert^2+\vert a_{0,1}\vert^2\ge 1$. Thus it follows from (\ref{2-12}) that
\begin{eqnarray*}
\sum_{k=2}^{\alpha-1}\left|a_{0, k}\right|^2+\sum_{k=1}^{\alpha-1}\left|a_{1, k}\right|^2+\sum_{k=0}^{\alpha-1} \sum_{n=2}^{\infty}\left|a_{n, k}\right|^2 &\leq M^2-\left(\left|a_{0,1}\right|^2+\left|a_{1,0}\right|^2\right) &\leq M^2-1,
\end{eqnarray*}
and hence (\ref{2-8}) holds. The proof of the theorem is complete.
\end{proof}

From Theorem \ref{thm1}, we have the following corollary.

\begin{corollary}\label{cor2}
Suppose that $\alpha \geq 2$ and $F(z)=\sum_{k=0}^{\alpha-1} \bar{z}^k A_k(z)$ is a polyanalytic function with $F(0)=0$, $|F(z)|\leq 1$ in $\mathbb{D}$, and all its non-zero coefficients $a_{n,k_1}, a_{n,k_2}$ satisfy \eqref{2-3}.
\begin{enumerate}
\item[{\rm (i)}]
 If $J_F(0)=1$, then $F(z)=\alpha z$, where $|\alpha|=1$. Also, if $J_F(0)=-1$, then $F(z)=\beta \overline{z}$, where $|\beta|=1$.
 \item[{\rm (ii)}]
If $\lambda_F(0)=1$, then either $F(z)=\gamma z$ or $F(z)=\gamma \overline{z}$, where $|\gamma|=1$.
\end{enumerate}
\end{corollary}

\begin{proof}(1) If $J_F(0)=1$, then it follows from (\ref{2-6}) that $ a_{n, k} =0$,  for each $(n, k) \neq(1,0),(0,1)$.
From (\ref{2-15}), we get that $ a_{0, 1} =0$ and thus, $\left|a_{1,0}\right|=1$. Then $F(z)=\alpha z$, where $|\alpha|=1$.

If $J_F(0)=-1$, then as above, we have $ a_{1, 0} =0$ and $\left|a_{0,1}\right|=1$ showing that $F(z)=\beta \overline{z}$, where $|\beta|=1$.

(2) If $\lambda_F(0)=1$, then, it follows from (\ref{2-8}) that $ a_{n, k} =0$  for each $(n, k) \neq(1,0),(0,1)$, which gives that $F(z)=\gamma z$ or $F(z)=\gamma \overline{z}$, where $|\gamma|=1$.
\end{proof}

\section{Landau-type theorems for polyanalytic 
functions}\label{HLP-sec3}

In this section, we establish three versions of  Landau-type theorems  for normalized and bounded polyanalytic functions in the unit disk.

\begin{theorem}\label{thm3}
Suppose that $M\geq\sqrt{\alpha}~(\geq 1)$,  and $F(z)=\sum_{k=0}^{\alpha-1} \bar{z}^k A_k(z)$ is a polyanalytic function satisfying the conditions of Theorem \ref{thm1}, where $A_k\in\mathcal{A}$ for each $k\in\{0,\ldots, \alpha-1\}$. If $|F(z)|\leq M$ for all $z\in\mathbb{D}$, then $F$ is univalent in the disk $\mathbb{D}_{r_1}$ and $F\left(\mathbb{D}_{r_1}\right)$ contains a schlicht disk $\mathbb{D}_{\sigma_1}$, where $r_1$ is the least positive root of the equation
\begin{eqnarray} \label{phi_r}
2- \frac{1-(\alpha +1)r^{\alpha}+\alpha r^{\alpha +1}}{(1-r)^2} -  \frac{r\, \sqrt{M^2-\alpha}}{(1-r)^3} \Big[2+ r^{\alpha}\{\alpha r-(\alpha +2)\}\Big] =0
\end{eqnarray}
and
\begin{eqnarray*}
\sigma_1=r_1-\frac{r_1^2-r_1^{\alpha+1}}{1-r_1}-\sqrt{M^2-\alpha}\, \frac{(1-r_1^\alpha)r_1^2}{(1-r_1)^2}.
\end{eqnarray*}
In particular, if $|F(z)|\leq 1$ for all $z\in\mathbb{D}$, then $F$ is univalent and maps $\mathbb{D}$ onto $\mathbb{D}$.
\end{theorem}

\begin{proof} We first prove that $F$ is univalent in the disk $\mathbb{D}_{r_1}$. By the assumption on $A_k$, we have $A_k(0)=0$ and $A_k^{\prime}(0)=1$, and thus, $a_{0, k}=0$  and $a_{1, k}=1$ for $k \in\{0, \ldots, \alpha-1\}.$

A direct calculation shows that
\begin{equation}
F_z(z)=\sum_{k=0}^{\alpha-1} \overline{z}^k A_k^{\prime }(z)=A_0^\prime(z)+\sum_{k=1}^{\alpha-1} \overline{z}^k \sum_{n=1}^{\infty} n a_{n, k} z^{n-1},
\label{liu1}
\end{equation}
and
\begin{equation}
F_{\bar{z}}(z)=\sum_{k=1}^{\alpha-1} k \overline{z}^{k-1} A_k(z)=\sum_{k=1}^{\alpha-1} k \overline{z}^{k-1} \sum_{n=0}^{\infty} a_{n, k} z^n.
\label{liu2}
\end{equation}

Note that $F_z(0)=A_0^{\prime }(0)=a_{1,0}=1,$ and $F_{\bar{z}}(0)=A_1(0)=a_{0,1}=0$. Then for any $z_1 \neq z_2$, where $z_1, z_2 \in \mathbb{D}_r$ and $r \in(0,1)$, by (\ref{2-5}), we have

\vspace{6pt}
$ |F(z_1)-F(z_2)|$
\begin{eqnarray*}
&=&\left|\int_{[z_1, z_2]} F_z d z+F_{\bar{z}} d \bar{z}\right| \\
& = &\left|\int_{\left[z_1, z_2\right]} F_z(0) d z+F_z(0) d \bar{z}+\int_{[z_1, z_2]}\left(F_z(z)-F_z(0)\right) d z+\left(F_{\bar{z}}(z)-F_{\bar{z}}(0)\right) d \bar{z}\right| \\
& \geqslant &\left|\int_{[z_1, z_2]} d z\right|-\left|\int_{[z_1, z_2]}\left(A_0^{\prime}(z)-A_0^{\prime}(0)\right) d z\right|-\left|\int_{[z_1, z_2]} \sum_{k=1}^{\alpha-1}\left[\bar{z}^k A_k^{\prime}(z) d z+k \bar{z}^{k-1} A_k(z)d \bar{z}\right] \right|  \\
& \geqslant &\left|z_2-z_1\right|-\int_{\left[z_1, z_2\right]} \sum_{n=2}^{\infty} n\left|a_{n,0}\right| r^{n-1}|d z|-\int_{\left[z_1, z_2\right]} \sum_{k=1}^{\alpha-1}\left|\bar{z}^k\right|\left(\Big|A_k^{\prime}(z)\Big|+k\Big|\frac{A_k(z)}{z}\Big|\,\right) |d z| \\
& \geqslant &\left|z_2-z_1\right|\left(1-\sum_{n=1}^{\infty}(n+1)\left|a_{n+1, 0}\right| r^n-\sum_{k=1}^{\alpha-1} r^k \sum_{n=1}^{\infty}\left(n\left|a_{n, k}\right|+k\left|a_{n, k}\right|\right) r^{n-1}\right) \\
\end{eqnarray*}\begin{eqnarray*}
& = &\left|z_2-z_1\right|\left\{1-\sum_{n=1}^{\infty}(n+1)\left|a_{n+1, 0}\right| r^n-\sum_{k=1}^{\alpha-1} r^k\left[1+k+\sum_{n=2}^{\infty}\left(n\left|a_{n, k}\right|+k\left|a_{n, k}\right|\right) r^{n-1}\right]\right\} \\
& \geqslant &\left|z_2-z_1\right|\left\{1-\sum_{n=1}^{\infty} \sqrt{M^2-\alpha}\,(n+1) r^n-\sum_{k=1}^{\alpha -1} r^k\left[1+k+\sum_{n=2}^{\infty} \sqrt{M^2-\alpha}\,(n+k) r^{n-1}\right]\right\} \\
&\geq&\left|z_2-z_1\right| \varphi(r),
\end{eqnarray*}
where
\be\label{ex-eq5}
\varphi(r)= 1-\sqrt{M^2-\alpha} \,\frac{2 r-r^2}{(1-r)^2}-\sum_{k=1}^{\alpha-1} r^k\left[1+k+\frac{(k+2) r-(k+1) r^2}{(1-r)^2} \sqrt{M^2-\alpha}\right] .
\ee
We now simplify the last expression and claim that $\varphi(r)=0$ is equivalent to \eqref{phi_r}.
In order to prove this claim, we use the following well-known identities:
\be\label{ex-eq1}
 \sum_{k=0}^{m+1}r^k =\frac{1-r^{m+2}}{1-r}
\ee
and by differentiating the last relation with respect to $r$ leads to
\be\label{ex-eq2}
 \sum_{k=1}^{m}(k+1)r^k =\frac{1-(m+2)r^{m+1}+(m+1)r^{m+2}}{(1-r)^2}-1.
\ee
In particular,  $m=\alpha -1$ leads to
\be\label{ex-eq3}
 \sum_{k=1}^{\alpha -1}(k+2)r^{k+1} =  \sum_{k=1}^{\alpha}(k+1)r^{k}  -2r = \frac{1-(\alpha+2)r^{\alpha+1}+(\alpha+1)r^{\alpha+2}}{(1-r)^2}-1-2r.
\ee
In view of \eqref{ex-eq1}, \eqref{ex-eq2} and \eqref{ex-eq3}, the function $\varphi(r)$ defined by \eqref{ex-eq5} takes the form
\begin{eqnarray*}
\varphi(r)&=& 2-\frac{1-(\alpha+1)r^{\alpha}+\alpha r^{\alpha+1}}{(1-r)^2}   \\
&& \hspace{1cm} - \frac{\sqrt{M^2-\alpha}}{(1-r)^2} \left [2 r-r^2 + \frac{1-(\alpha+2)r^{\alpha+1}+(\alpha+1)r^{\alpha+2}}{(1-r)^2} \right .\\
&& \hspace{5cm}  \left . -1-2r -r^2 \left ( \frac{1-(\alpha+1)r^{\alpha}+\alpha r^{\alpha+1}}{(1-r)^2}-1\right )\right ]\\
&=& 2-\frac{1-(\alpha+1)r^{\alpha}+\alpha r^{\alpha+1}}{(1-r)^2}\\
&& - \frac{\sqrt{M^2-\alpha}}{(1-r)^4} [
 -(1-r)^2 + 1-r^2-(\alpha+2)r^{\alpha+1}+2(\alpha+1)r^{\alpha+2}- \alpha r^{\alpha+3} ].
\end{eqnarray*}
As $-(\alpha+2)r^{\alpha+1}+2(\alpha+1)r^{\alpha+2}- \alpha r^{\alpha+3}= (1-r)(\alpha r -(\alpha+2)) r^{\alpha+1}$, it follows from elementary calculations that
$$
\varphi(r)= 2-\frac{1-(\alpha+1)r^{\alpha}+\alpha r^{\alpha+1}}{(1-r)^2}
  -  \sqrt{M^2-\alpha}\,\frac{r}{(1-r)^3} [2 + (\alpha r -(\alpha+2)) r^{\alpha} ]
$$
and the claim is proved. To this end, it is easy to verify that there exists an $\varepsilon>0$ such that the function $\varphi(r)$ is continuous for $r \in[0,1-\varepsilon]$, and
$$\lim _{r \rightarrow 0} \varphi(r)=1>0 ~\mbox{ and }~ \lim _{r \rightarrow 1-\varepsilon} \varphi(r)<0.
$$
Therefore, according to the zero-point existence theorem, there exists an $r_1 \in(0,1)$ such that $\varphi(r_1)=0$. This implies that
$\left|F(z_1)-F(z_2)\right| >0$ for two distinct points $z_1, z_2 \in \mathbb{D}_{r_1}$ and hence, $F$ is univalent in $\mathbb{D}_{r_1}$.

By Theorem \ref{thm1}(1), we see that $\left|a_{n, k}\right| \leq \sqrt{M^2-\alpha}\,$ for $n \geq 2$ and $k=0, 1, \ldots , \alpha -1$. Therefore, for any $w\in\partial\mathbb{D}_{r_1}$, we obtain
\begin{eqnarray*}
|F(w)-F(0)|&=&\left|\sum_{k=0}^{\alpha-1} \bar{w}^k A_k(w)\right|
=\left|\sum_{k=0}^{\alpha-1} \bar{w}^k\left(a_{1, k} w+\sum_{n=2}^{\infty} a_{n k} w^n\right)\right| \\
& =&\left|\sum_{k=0}^{\alpha-1} \bar{w}^k a_{1, k} w+\sum_{k=0}^{\alpha-1} \bar{w}^k \sum_{n=2}^{\infty} a_{n, k} w^n\right| \\
& =&\left|a_{1,0} w+\sum_{k=1}^{\alpha-1} a_{1, k}\left| w\right|^2 \overline{w}^{k-1}+\sum_{k=0}^{\alpha-1}\bar{w}^k \sum_{n=2}^{\infty} a_{n, k} w^n \right| \\ & \geq & r_1-\sum_{k=1}^{\alpha-1} r_1^{k+1}-\sqrt{M^2-\alpha}\, \sum_{k=0}^{\alpha-1} \sum_{n=2}^{\infty} r_1^{n+k} \\
& =& r_1-\frac{r_1^2\left(1-r_1^{\alpha-1}\right)}{1-r_1}-\sqrt{M^2-\alpha}\, \frac{(1-r_1^\alpha)r_1^2}{(1-r_1)^2}  =\sigma_1.
\end{eqnarray*}
Hence, $F\left(\mathbb{D}_{r_1}\right)$ contains a schlicht disk $\mathbb{D}_{\sigma_1}$. See Tables \ref{tab1} and \ref{tab2}.

If $|F(z)|\leq 1$ for all $z\in\mathbb{D}$, then, as $F(0)=0$ and $\lambda_F(0)=\big||F_z(0)|-|F_{\overline{z}}(0)|\big|=1$, by Corollary \ref{cor2}, we obtain that $F(z)=\gamma z$ or $F(z)=\gamma \overline{z}$, where $|\gamma|=1$. It follows that $F$ is univalent and maps $\mathbb{D}$ onto $\mathbb{D}$.
\end{proof}

\begin{table}[h!]
\begin{tabular*}{15cm}[]{p{1.3cm}p{1.3cm}p{1.3cm}p{1.3cm}p{1.3cm}p{1.3cm}p{1.3cm}p{1.3cm}p{1.3cm}c}
\hline   $M$ & $1.5$  & $2$ & $2.5$ & $3$ &
$3.5$ & $4$ & $4.5$ & $5$\\
\hline
$r_1$ & 0.2450 & 0.1550 & 0.1120 & 0.0875 & 0.0712 & 0.0598 & 0.0515 & 0.0452 \\
$\sigma_1$ & 0.0133 & 0.0846 & 0.0572 & 0.0410 & 0.0312 & 0.0246 & 0.0199 & 0.0164 \\
\hline
\end{tabular*}

\vspace{6pt}
\caption{The values of $r_1$ and $\sigma_1$ in  Theorem \ref{thm3} for $\alpha=2$. \label{tab1}}
\end{table}

\begin{table}[h!]
\begin{tabular*}{15cm}[]{p{1.3cm}p{1.3cm}p{1.3cm}p{1.3cm}p{1.3cm}p{1.3cm}p{1.3cm}p{1.3cm}p{1.3cm}c}
\hline   $M$  & $2$ & $2.5$ & $3$ &
$3.5$ & $4$ & $4.5$ & $5$ & $5.5$\\
\hline
$r_1$ & 0.1600 & 0.1140 & 0.0880 & 0.0710 & 0.0590 & 0.0505 & 0.0370 & 0.0340\\
$\sigma_1$ & 0.0940 & 0.0600 & 0.0425 & 0.0320 & 0.0252 & 0.0205 & 0.0155 & 0.0120\\
\hline
\end{tabular*}

\vspace{6pt}
\caption{The values of $r_1$ and $\sigma_1$ in  Theorem \ref{thm3} for $\alpha=3$. \label{tab2}}
\end{table}

\begin{theorem}\label{thm4}
Suppose that $M\geq 1$, $\alpha\geq 3$, and $F$ is a polyanalytic function satisfying the conditions of Theorem \ref{thm1}. If $|F(z)|\leq M$ for all $z\in\mathbb{D}$, and $F(0)=\left|J_F(0)\right|-1=0$, then $F$ is univalent in the disk $\mathbb{D}_{r_2}$ and $F\left(\mathbb{D}_{r_2}\right)$ contains a schlicht disk  $\mathbb{D}_{\sigma_2}$, where $\lambda_0(M)$ is defined by \eqref{2-7} and $r_2$ is the least positive root of the equation
\begin{equation}\label{psi_r}
\lambda_0(M)+\sqrt{M^2-1} - \frac{\sqrt{M^2-1}}{(1-r)^3}\Big[r(2-r)(1-r)+ 1+r+ r^{\alpha -1}( (\alpha -2)r- \alpha)\Big] =0
\end{equation}
and
\begin{equation}\label{sgma_2}
 \sigma_2=  r_2\lambda_0(M)- \sqrt{M^2-1}\left( \frac{2r_2^2(1-r_2^{\alpha -2})}{1-r_2}+r_2^{\alpha}+\frac{r_2^2(1-r_2^{\alpha})}{(1-r_2)^2}
  \right).
\end{equation}
In particular, if $|F(z)|\leq 1$ for all $z\in\mathbb{D}$ and $F(0)=\left|J_F(0)\right|-1=0$, then $F$ is univalent and maps $\mathbb{D}$ onto  $\mathbb{D}$.
\end{theorem}

\begin{proof}
From (\ref{liu1}) and (\ref{liu2}), we have that $F_z(0)=A_0^{\prime }(0)=a_{1,0}$ and $F_{\bar{z}}(0)=A_1(0)=a_{0,1}$. Then for any $z_1 \neq z_2$, where $z_1, z_2 \in \mathbb{D}_r$ and $r \in(0,1)$ is a constant,  we have
$$\begin{aligned}
&|F(z_1)-F(z_2)|=\left|\int_{[z_1, z_2]} F_z d z+F_{\bar{z}} d \bar{z}\right| \\
& \geq\left|\int_{\left[z_1, z_2\right]} F_z(0) d z+F_{\bar{z}}(0) d \bar{z}\right|-\left|\int_{\left[z_1, z_2\right]}\left(F_z(z)-F_z(0)\right) d z+\left(F_{\bar{z}}(z)-F_{\bar{z}}(0)\right) d \bar{z}\right|\\
&=\left|\int_{\left[z_1, z_2\right]} F_z(0) d z+F_{\bar{z}}(0) d \bar{z}\right|\\&-\left|\int_{\left[z_1, z_2\right]}\left(\sum_{k=0}^{\alpha-1} \bar{z}^k A_k^{\prime}(z)-A_0^{\prime }(0)\right) d z+\left(\sum_{k=1}^{\alpha-1} k \bar{z}^{k-1} A_k(z)-A_1(0)\right) d \bar{z}\right|
 \\
& \geq J_1-J_2-J_3-J_4-J_5,
\end{aligned}$$
where 
\begin{eqnarray*}
J_1 & :=&\left|\int_{\left[z_1, z_2\right]} F_z(0) d z+F_{\bar{z}}(0) d \bar{z}\right|,\quad J_2  :=\left|\int_{\left[z_1, z_2\right]}\left(A_0^{\prime}(z)-A_0^{\prime}(0)\right) d z\right|, \\
J_3 & :=&\left|\int_{\left[z_1, z_2\right]}\left(A_1(z)-A_1(0)\right) d z\right|, \quad J_4  :=\left|\int_{\left[z_1, z_2\right]} \sum_{k=1}^{\alpha-1} \bar{z}^k A_k^{\prime}(z) d z\right|,~\mbox{ and} \\
J_5 & :=&\left|\int_{\left[z_1, z_2\right]} \sum_{k=2}^{\alpha-1} k\bar{z}^{k-1} A_k(z) d \bar{z}\right| .
\end{eqnarray*}

At first, by Theorem \ref{thm1}(2), we find that $M\geq 1$, and
$$
\lambda_F(0) \geq \lambda_0(M)=\frac{\sqrt{2}}{\sqrt{M^2-1}\,+\sqrt{M^2+1}},
$$
so that,
$$
J_1 \geq\int_{\left[z_1, z_2\right]} \big|\left|a_{1,0}\right|-\left|a_{0,1}\right|\big| \left|d z \right|=\int_{\left[z_1, z_2\right]} \lambda_F(0)|d z|
\geq \lambda_0(M)\left|z_1-z_2\right|.$$

Next, by Theorem \ref{thm1}(2), we see that $\left|a_{n, k}\right| \leq \sqrt{M^2-1}$ for each $(n, k) \neq(1,0),(0,1)$, and thus, in order to estimate $J_p$ for $p=2 \text{ to }5$, we first compute
\begin{eqnarray*}
 \left|A_0^{\prime}(z)-A_0^{\prime}(0)\right| &\leq& \sum_{n=2}^{\infty} n\left|a_{n,0}\right| r^{n-1}
\leq  \sqrt{M^2-1}\, \sum_{n=2}^{\infty} n r^{n-1} =\sqrt{M^2-1}\, \frac{2 r-r^2}{(1-r)^2} ,\\
  \left|A_1(z)-A_1(0)\right| &\leq& \sum_{n=1}^{\infty} \left|a_{n,1}\right| r^n  \leq  \sqrt{M^2-1}\, \frac{r}{1-r} ,\\
 \left|\sum_{k=1}^{\alpha-1} \bar{z}^k A_k^{\prime}(z)\right|
&\leq &  \sum_{k=1}^{\alpha-1}|\bar{z}|^k \sum_{n=1}^{\infty} n\left|a_{n, k}\right||z|^{n-1}
\leq \sqrt{M^2-1}\, \sum_{k=1}^{\alpha-1} \frac{r^k}{(1-r)^2}, ~\mbox{ and }\\
 \left| \sum_{k=1}^{\alpha-1} k\bar{z}^{k-1} A_k(z)\right|
&\leq&  \sum_{k=2}^{\alpha-1}k|\bar{z}|^{k-1} \sum_{n=0}^{\infty} \left|a_{n, k}\right||z|^n \leq
\sqrt{M^2-1}\, \sum_{k=2}^{\alpha-1}  \frac{kr^{k-1}}{1-r}.
\end{eqnarray*}

Using these estimates, one can easily see that
$$\left|F\left(z_1\right)-F\left(z_2\right)\right| \geq J_1-\sum_{p=2}^{5}J_p \geq \left|z_1-z_2\right| \psi(r),$$
where $\psi(r)$ is defined by
\begin{equation}\label{ex-eq4}
\psi(r)=\lambda_0(M)-\sqrt{M^2-1}\left(\frac{2 r-r^2}{(1-r)^2}+\frac{r}{1-r}+\sum_{k=1}^{\alpha-1} \frac{r^k}{(1-r)^2}+\sum_{k=2}^{\alpha-1} \frac{k r^{k-1}}{1-r}\right),
\end{equation}

In order to show that \eqref{psi_r} is equivalent to $\psi(r)=0$,  by \eqref{ex-eq1} and \eqref{ex-eq2},  we simplify the bracketed term in \eqref{ex-eq4}
as
\begin{eqnarray*}
B&=& \frac{2 r-r^2}{(1-r)^2}+\frac{r}{1-r}+\frac{1}{(1-r)^2}\left ( \frac{r(1-r^{\alpha -1})}{1-r} \right )
+\frac{1}{1-r}\left ( \frac{1-\alpha r^{\alpha -1}+(\alpha -1) r^{\alpha}}{(1-r)^2} -1 \right )\\
&=& \frac{2 r-r^2}{(1-r)^2}-1 +\frac{1}{(1-r)^3}[ r(1-r^{\alpha -1})+ 1-\alpha r^{\alpha-1}+(\alpha -1) r^{\alpha} ]\\
&=& \frac{2 r-r^2}{(1-r)^2}-1 +\frac{1}{(1-r)^3}[1+r+ r^{\alpha -1}( (\alpha -2)r- \alpha)]
\end{eqnarray*}
which shows that
$$\psi(r)= \lambda_0(M) +\sqrt{M^2-1} -  \frac{\sqrt{M^2-1}}{(1-r)^3}[r(2-r)(1-r)+ 1+r+ r^{\alpha -1}( (\alpha -2)r- \alpha)]
$$
and thus,  $\psi(r)=0$ is same as \eqref{psi_r}. Furthermore, the function $\psi(r)$ is strictly decreasing for $r \in(0,1)$ and $M>1$ ,
$$
\lim _{r \rightarrow 0+} \psi(r)=\lambda_0(M)=\frac{\sqrt{2}}{\sqrt{M^2-1}\,+\sqrt{M^2+1}}>0 \quad \text { and } \quad \lim _{r \rightarrow 1^{-}} \psi(r)=-\infty.
$$

Hence there exists a unique $r_2\in(0,1)$ satisfying $\psi(r_2)=0$. This implies that $F$ is univalent in $\mathbb{D}_{r_2}$.
Next, we consider
\begin{eqnarray*}
\sum_{k=0}^{\alpha-1} \bar{w}^k A_k(w)
&=&\sum_{k=0}^{\alpha-1} \bar{w}^k\left(a_{0, k}+a_{1, k} w+\sum_{n=2}^{\infty} a_{n, k} w^n\right)\\
& =& a_{0,0}+a_{0,1} \bar{w}+\sum_{k=2}^{\alpha-1} a_{0, k}\bar{w}^k +a_{1,0} w+\sum_{k=1}^{\alpha-1} a_{1, k} w\bar{w}^k +\sum_{k=0}^{\alpha-1} \bar{w}^k \sum_{n=2}^{\infty} a_{n, k} w^n
\end{eqnarray*}
so that for any $w\in\partial\mathbb{D}_{r_2}$, we obtain (since $a_{0,0}=0$)
\begin{eqnarray*}
|F(w)-F(0)| & \geqslant &\left|a_{0,1} \bar{w}+a_{1,0} w\right|-\left|\sum_{k=2}^{\alpha-1}a_{0, k} \bar{w}^k \right|-\left|\sum_{k=1}^{\alpha-1} a_{1, k} w\bar{w}^k \right|-\left|\sum_{k=0}^{\alpha-1} \bar{z}^k \sum_{n=2}^{\infty} a_{n, k} w^n\right|\\
&\geqslant &\lambda_F(0) r_2-\sqrt{M^2-1} \left (\sum_{k=2}^{\alpha-1} r_2^k -\sum_{k=1}^{\alpha-1} r_2^{k+1}-\sum_{k=0}^{\alpha-1} \sum_{n=2}^{\infty}r_2^{n+k} \right ) \\
&\geq & r_2\left[\lambda_0(M)- \sqrt{M^2-1}\left(\sum_{k=2}^{\alpha-1} r_2^{k-1}+\sum_{k=1}^{\alpha-1} r_2^k+\sum_{k=0}^{\alpha-1} \frac{r_2^{1+k}}{1-r_2}\right)\right] =:\sigma_2.
\end{eqnarray*}

We also have (as $\psi(r_2)=0$)
\begin{eqnarray*}
\sigma_2 
& > &r_2\left[\lambda_0(M)-\sqrt{M^2-1}\left(\frac{2 r_2-r_2^2}{(1-r_2)^2}+\frac{r_2}{1-r_2}+\sum_{k=1}^{\alpha-1} \frac{r_2^k}{(1-r_2)^2}+\sum_{k=2}^{\alpha-1} \frac{k r_2^{k-1}}{1-r_2}\right)\right]\\
& = & r_2\psi(r_2)=0
\end{eqnarray*}
and it is a simple exercise to see that (for $\alpha \ge 3$)
\begin{eqnarray*}
\sum_{k=2}^{\alpha-1} r^{k}+\sum_{k=1}^{\alpha-1} r^{k+1}+\sum_{k=0}^{\alpha-1} \frac{r^{2+k}}{1-r}
&=&2r^2\sum_{k=0}^{\alpha-3} r^{k} +r^{\alpha}+ \frac{r^2}{1-r}\sum_{k=0}^{\alpha-1} r^{k}\\
&=& \frac{2r^2(1-r^{\alpha -2})}{1-r}+r^{\alpha}+\frac{r^2(1-r^{\alpha})}{(1-r)^2}
\end{eqnarray*}
which gives  $\sigma_2$ given by \eqref{sgma_2}.
Hence, $F\left(\mathbb{D}_{r_2}\right)$ contains a schlicht disk $\mathbb{D}_{\sigma_2}.$

If $|F(z)|\leq 1$ for all $z\in\mathbb{D}$ and $F(0)=\left|J_F(0)\right|-1=0$, then by Corollary \ref{cor2}, $F$ is univalent and maps $\mathbb{D}$ onto  $\mathbb{D}.$
\end{proof}

For $\alpha=2$ or 3, we can obtain the following corollaries from Theorem \ref{thm4}.

\begin{corollary}\label{cor3}
Suppose that $M\geq 1$, and let $F(z)=A_0(z)+\bar{z} A_1(z)$ be a polyanalytic function of order $2$ on $\mathbb{D}$ satisfying the conditions of Theorem \ref{thm1}. If $|F(z)|\leq M$ for all $z\in\mathbb{D}$, and $F(0)=\left|J_F(0)\right|-1=0$, then $F$ is univalent in the disk $\mathbb{D}_{r_2^{\prime}}$ and $F\left(\mathbb{D}_{r_2^{\prime}}\right)$ contains a schlicht disk  $\mathbb{D}_{\sigma_2^{\prime}}$, where $\lambda_0(M)$ is defined by \eqref{2-7} and $$r_2^{\prime}=1-\sqrt{\frac{2\sqrt{M^2-1}}{\lambda_0(M)+2\sqrt{M^2-1}}}
$$
and
$$\sigma_2^{\prime}=\lambda_0(M) r_2^{\prime}-\sqrt{M^2-1} \frac{2 {r_2^{\prime}}^2}{1-r_2^{\prime}}.
$$
In particular, if $|F(z)|\leq 1$ for all $z\in\mathbb{D}$ and $F(0)=\left|J_F(0)\right|-1=0$, then $F$ is univalent and maps $\mathbb{D}$ onto  $\mathbb{D}$.
\end{corollary}
\begin{proof} The number $r_2^{\prime}$ can be easily obtained from the proof of Theorem \ref{thm4} as the least positive root of the following equation:
$$ \lambda_0(M)-\sqrt{M^2-1}\, \frac{2r(2-r)}{(1-r)^2}=0.
$$
Solving this equation for $r$ gives the value $r_2^{\prime}$.

Next, we prove the covering radius $\sigma_2^{\prime}$.
For any $w\in\partial\mathbb{D}_{r_2^{\prime}}$, we obtain
$$
\begin{aligned}
|F(w)-F(0)| & =\left|A_0(w)+\bar{w} A_1(w)\right| \\
& =\left|\sum_{n=0}^{\infty} a_{n, 0} w^n+\bar{w} \sum_{n=0}^{\infty} a_{n, 1} w^n\right| \\
& =\left|a_{0,0}+a_{1,0} w+a_{0,1} \bar{w}+\sum_{n=2}^{\infty} a_{n, 0} w^n+\bar{w} \sum_{n=1}^{\infty} a_{n, 1} w^n\right| \\
& \geqslant\left|a_{1,0} w+a_{0,1} \bar{w}\right|-\left|\sum_{n=2}^{\infty} a_{n, 0} w^n\right|-\left|\sum_{n=1}^{\infty} a_{n, 1} \bar{w} w^n\right| \\
& \geqslant \lambda_F(0) r_2^{\prime}-\sqrt{M^2-1} \sum_{n=2}^{\infty} {r_2^{\prime}} ^n-\sqrt{M^2-1} \sum_{n=1}^{\infty} {r_2^{\prime}} ^{n+1} \\
& \geqslant \lambda_0(M) r_2^{\prime}-\sqrt{M^2-1} \frac{2 {r_2^{\prime}} ^2}{1-r_2^{\prime}}:=\sigma_2^{\prime}.
\end{aligned}
$$
Hence, $F\left(\mathbb{D}_{r_2^{\prime}}\right)$ contains a schlicht disk $\mathbb{D}_{\sigma_2^{\prime}}.$ See Table \ref{tab3}.
\end{proof}

\begin{table}[h!]
\begin{tabular*}{15cm}[]{p{1.3cm}p{1.3cm}p{1.3cm}p{1.3cm}p{1.3cm}p{1.3cm}p{1.3cm}p{1.3cm}p{1.3cm}c}
\hline   $M$ & $1.5$  & $2$ & $2.5$ & $3$ &
$3.5$ & $4$ & $4.5$ & $5$\\
\hline
$r_2^{\prime}$ & 0.4562 & 0.3028 & 0.2285 & 0.1872 & 0.1596 & 0.1399 & 0.1251 & 0.1135 \\
$\sigma_2^{\prime}$ & 0.2078 & 0.1095 & 0.0689 & 0.0486 & 0.0363 & 0.0284 & 0.0228 & 0.0189 \\
\hline
\end{tabular*}

\vspace{6pt}
\caption{The values of $r_2^{\prime}$ and $\sigma_2^{\prime}$ in  Corollary \ref{cor3} for $\alpha=2$.\label{tab3}}
\end{table}

\begin{corollary}\label{cor4}
Suppose that $M\geq 1$, and let $F(z)=A_0(z)+\bar{z} A_1(z)+\bar{z}^2 A_2(z)$ be a polyanalytic function of order $3$ on $\mathbb{D}$ satisfying the conditions of Theorem \ref{thm1}. If $|F(z)|\leq M$ for all $z\in\mathbb{D}$, and $F(0)=\left|J_F(0)\right|-1=0$, then $F$ is univalent in the disk $\mathbb{D}_{r_2^*}$ and $F\left(\mathbb{D}_{r_2^*}\right)$ contains a schlicht disk  $\mathbb{D}_{\sigma_2^*}$, where $\lambda_0(M)$ is defined by \eqref{2-7} and
$$r_2^*=1-\sqrt{\frac{3\sqrt{M^2-1}}{\lambda_0(M)+3\sqrt{M^2-1}}}
$$
and
$$\sigma_2^*=\lambda_0(M) r_2^*-\sqrt{M^2-1} \frac{3{r_2^*}^2}{1-r_2^*}.$$
In particular, if $|F(z)|\leq 1$ for all $z\in\mathbb{D}$ and $F(0)=\left|J_F(0)\right|-1=0$, then $F$ is univalent and maps $\mathbb{D}$ onto  $\mathbb{D}$.
\end{corollary}
\begin{proof}
The proof may be directly obtained from Theorem \ref{thm4}.  Indeed, the number $r_2^*$ is the least positive root of the equation
$$ \lambda_0(M)-\sqrt{M^2-1}\, \frac{3r(2-r)}{(1-r)^2}=0
$$
which follows from Theorem \ref{thm4} for the case $\alpha =3$. See Table \ref{tab4}.
\end{proof}

\begin{table}[h!]
\begin{tabular*}{15cm}[]{p{1.3cm}p{1.3cm}p{1.3cm}p{1.3cm}p{1.3cm}p{1.3cm}p{1.3cm}p{1.3cm}p{1.3cm}c}
\hline   $M$ & $1.5$  & $2$ & $2.5$ & $3$ &
$3.5$ & $4$ & $4.5$ & $5$\\
\hline
$r_2^*$ & 0.3817 & 0.2403 & 0.1732 & 0.1348 & 0.1102 & 0.0935 & 0.0814 & 0.0720 \\
$\sigma_2^*$ & 0.1458 & 0.0693 & 0.0398 & 0.0255 & 0.0178 & 0.0130 & 0.0099 & 0.0078 \\
\hline
\end{tabular*}

\vspace{6pt}
\caption{For $\alpha=3$: The values of $r_2^*$ and $\sigma_2^*$ in  Corollary \ref{cor4}. \label{tab4}}
\end{table}

If we replace the condition $|J_F(0)|-1=0$ with $\lambda_F(0)-1=0$ in Theorem \ref{thm4} and apply the analogous method of the proof of Theorem \ref{thm4}, we may verify the following theorem which we state without proof.

\begin{theorem}\label{thm5} Suppose that $M>1$, $\alpha\geq 3$ and $F$ is a polyanalytic function
satisfying the conditions of Theorem \ref{thm1}, If $|F(z)|\leq M$ for all $z\in\mathbb{D}$, and $F(0)=\lambda_F(0)-1=0$, then $F$ is univalent in the disk $\mathbb{D}_{r_3}$ and $F\left(\mathbb{D}_{r_3}\right)$ contains a schlicht disk $\mathbb{D}_{\sigma_3}$, where $r_3$ is the least positive root of the following equation \eqref{psi_r} with $\lambda_0(M)=1$, i.e.,
\begin{equation*}
1-\sqrt{M^2-1} + \frac{\sqrt{M^2-1}}{(1-r)^3}\Big[r(2-r)(1-r)+ 1+r+ r^{\alpha -1}( (\alpha -2)r- \alpha)\Big] =0,
\end{equation*}
and $\sigma_3$ is same as $\sigma_2$ but with $\lambda_0(M)=1$, i.e.,
\begin{equation*}
 \sigma_3=  r_3- \sqrt{M^2-1}\left( \frac{2r_3^2(1-r_3^{\alpha -2})}{1-r_3}+r_3^{\alpha}+\frac{r_3^2(1-r_3^{\alpha})}{(1-r_3)^2}  \right).
\end{equation*}


In particular, if $|F(z)|\leq 1$ for all $z\in\mathbb{D}$ and $F(0)=\lambda_F(0)-1=0$, then $F$ is univalent and maps $\mathbb{D}$ onto  $\mathbb{D}$.
\end{theorem}

\section{Landau-type theorems for $\log$-$\alpha$-analytic functions}\label{HLP-sec4}
In this section, we establish three versions of  Landau-type theorems  for certain normalized $\log$-$\alpha$-analytic functions in the unit disk by means of Theorems \ref{thm3}-\ref{thm5} and Lemma \ref{lem1}. In particular, we obtain the accurate value of Bloch constant for certain $\log$-$\alpha$-analytic mappings in $\mathbb{D}$.

\begin{theorem}\label{thm6}
Suppose that $M\geq \sqrt{\alpha}\, (\geq 1)$, and $f(z)=\prod_{k=0}^{\alpha-1}\left(a_k(z)\right)^{\bar{z}^k}$ is a $\log$-$\alpha$-analytic function on $\mathbb{D}$, satisfying $f(0)=1$, where for each $k$ , $a_k(z)$ is $\log$-analytic on $\mathbb{D}$, $a_k(0)=a_k^{\prime}(0)=1$ and $A_k(z):=\log a_k(z)=\sum_{n=0}^{\infty} a_{n, k} z^n$, and all its non-zero coefficients $a_{n_1,k_1}, a_{n_2,k_2}$ satisfy $\left|\arg \frac{a_{n_1, k_1}}{a_{n_2, k_2}}\right| \leq \frac{\pi}{2}$ for each $k_1,k_2\in\{0,\ldots , \alpha-1\}$ and $n_1, n_2\in\mathbb{N}$ with $k_1\neq k_2,\, n_1\neq n_2$. If $|\log f(z)| \leq M$, then $f$ is univalent in the disk $\mathbb{D}_{r_1}$ and $f\left(\mathbb{D}_{r_1}\right)$ contains a schlicht disk $\mathbb{D}\left(w_1, \mu_1\right)$, where $r_1$ and $\sigma_1$ are as in Theorem \ref{thm3} with
$w_1=\cosh \sigma_1$ and $\mu_1=\sinh \sigma_1$.

In particular, if $|\log f(z)|\leq 1$ for all $z\in\mathbb{D}$, then $f$ is univalent in $\mathbb{D}$ and $f\left(\mathbb{D}\right)$ contains a schlicht disk $\mathbb{D}\left(w_1', \mu_1'\right)$, where $w_1'=\cosh 1 \text{ and }\, \mu_1'=\sinh 1$. This result is sharp.
\end{theorem}

\begin{proof}
Set $F(z)=\log f(z)$. Then $F$ is a polyanalytic function of order $\alpha$ on the unit disk. Thus $F(z)=\sum_{k=0}^{\alpha-1} \bar{z}^k A_k(z)=\sum_{k=0}^{\alpha-1} \bar{z}^k \sum_{n=0}^{\infty} a_{n, k} z^n$, $F(0)=0, A_k(0)=0, A_k^{\prime}(0)=1,|F(z)| \leq M$, and  therefore, $F$ satisfies the conditions of Theorem \ref{thm3}.

We first prove the univalence of $f(z)$ in $\mathbb{D}_{r_1}$. Consider,  for any $z_1, z_2 \in \mathbb{D}_{r}\left(0<r<r_1,z_1 \neq z_2\right)$. Then it follows from the proof of Theorem \ref{thm3} that
$$\begin{aligned}
\left|\log f\left(z_1\right)-\log f\left(z_2\right)\right|  &=\left|F\left(z_1\right)-F\left(z_2\right)\right|=\left|\int_{\left[z_1, z_2\right]} F_z(z) d z+F_{\bar{z}}(z) d \bar{z}\right| \\
& \geqslant\left|z_2-z_1\right| \varphi(r),
\end{aligned}
$$
where $\varphi(r)$ is given by \eqref{ex-eq5}.
Thus $\log f\left(z_1\right) \neq \log f\left(z_2\right)$, which implies the univalency of $f$ in the disk $\mathbb{D}_{r_1}$.

Next, we consider any $z^{\prime}$ with $\left|z^{\prime}\right|=r_1$. As in the proof of Theorem \ref{thm3}, we get that
$$\begin{aligned}
\left|\log f\left(z^{\prime}\right)\right|
&=\left|F\left(z^{\prime}\right)\right|=\left|\sum_{k=0}^{\alpha-1} \bar{z}^{\prime k} A_k\left(z^{\prime}\right)\right|
& \geq r_1-\frac{r_1^2-r_1^{\alpha+1}}{1-r_1}-\sqrt{M^2-\alpha}\, \sum_{k=0}^{\alpha-1} \frac{r_1^{k+2}}{1-r_1}=\sigma_1.
\end{aligned}$$

By Lemma \ref{lem1}, we obtain that the range $f\left(\mathbb{D}_{r_1}\right)$ contains a schlicht disk $\mathbb{D}\left(w_1, \mu_1\right)$, where $w_1=\cosh \sigma_1 \text{ and } \mu_1=\sinh \sigma_1$.

If $|\log f(z)|\leq 1$ for all $z\in\mathbb{D}$, then it follows from Theorem \ref{thm3} and the above proof that $f$ is univalent in $\mathbb{D}$ and $f\left(\mathbb{D}\right)$ contains a schlicht disk $\mathbb{D}\left(w_1', \mu_1'\right)$, where $w_1'=\cosh 1 \text{ and } \mu_1'=\sinh 1$.

It is obvious that $r_1'=1$ is sharp. Now, we prove the sharpness of $\mu_1'$. To this end, we consider the $\log$-$p$-analytic function $f_1(z)=e^{F_1(z)}$. Here, $F_1(z)=\gamma z$ or $F_1(z)=\gamma \overline{z}$, where $|\gamma|=1$. It is easy to verify that $f_1(z)$ satisfies the hypothesis of Theorem \ref{thm6}, and thus, we have that $f_1(z)$ is univalent in $\mathbb{D}$, and the range $f_1\left( \mathbb{D}\right) $ contains a schlicht disk $U\left( w_1',r_1' \right)$.

Because the univalent radius $r_{1}'=1$ is sharp, the sharpness of the radius $\mu_1'=\sinh 1$ follows from Lemma \ref{lem1} and the fact that $\sigma_1'=1$. The proof of the theorem is complete.
\end{proof}

With the aid of Theorems \ref{thm4}, \ref{thm5} and Lemma \ref{lem1}, if we apply the analogous method as in the proof of Theorem \ref{thm6}, we may establish the following two new forms of the Landau-type theorems for certain $\log$-$\alpha$-analytic mappings. 

\begin{theorem}\label{thm7} Suppose that $M\geq 1$, $\alpha\geq 3$,  and $f(z)=\prod_{k=0}^{\alpha-1}\left(a_k(z)\right)^{\bar{z}^k}$ is a $\log$-$\alpha$-analytic function on $\mathbb{D}$, satisfying $f(0)=\left|J_f(0)\right|=1$, where for each $k$, $a_k(z)$ is $\log$-analytic on $\mathbb{D}$, and $A_k(z):=\log a_k(z)=\sum_{n=0}^{\infty} a_{n, k} z^n$, and all its non-zero coefficients $a_{n_1,k_1}, a_{n,k_2}$ satisfy $\left|\arg \frac{a_{n_1, k_1}}{a_{n_2, k_2}}\right| \leq \frac{\pi}{2}$ for each $k_1,k_2\in\{0,\ldots , \alpha-1\}$ and $n_1, n_2\in\mathbb{N}$ with $k_1\neq k_2,\, n_1\neq n_2$. If $|\log f(z)| \leq M$, then $f$ is univalent in the disk $\mathbb{D}_{r_2}$ and $f\left(\mathbb{D}_{r_2}\right)$ contains a schlicht disk $\mathbb{D}\left(w_2, \mu_2\right)$, where $r_2$ and $\sigma_2$ are as in Theorem  \ref{thm4} with $w_2=\cosh \sigma_2$ and $\mu_2=\sinh \sigma_2$.

In particular, if $|\log f(z)|\leq 1$ for all $z\in\mathbb{D}$ and $f(0)=\left|J_f(0)\right|=1$, then $f$ is univalent in $\mathbb{D}$ and $f\left(\mathbb{D}\right)$ contains a schlicht disk $\mathbb{D}\left(w_1', \mu_1'\right)$, where $w_1'=\cosh 1\text{ and } \mu_1'=\sinh 1$. This result is sharp.
\end{theorem}


\begin{corollary}\label{cor7a} Suppose that $M\geq 1$, and $f(z)=\left(a_0(z)\right) \cdot\left(\bar{z} a_1(z)\right)$ is a $\log$-$2$-analytic function on $\mathbb{D}$, satisfying $f(0)=\left|J_f(0)\right|=1$,  where $a_0(z), a_1(z)$ are $\log$-analytic on $\mathbb{D}$, and $A_k(z):=\log a_k(z)=\sum_{n=0}^{\infty} a_{n, k} z^n$ for $k=0,1$, and all its non-zero coefficients  $a_{n_1,0}, a_{n,1}$ satisfy $\left|\arg \frac{a_{n_1, 0}}{a_{n_2, 1}}\right| \leq \frac{\pi}{2}$ for each $n_1, n_2\in\mathbb{N}$ with $n_1\neq n_2$.  If $|\log f(z)| \leq M$, then $f$ is univalent in the disk $\mathbb{D}_{r_2^{\prime}}$ and $f\left(\mathbb{D}_{r_2^{\prime}}\right)$ contains a schlicht disk  $\mathbb{D}\left(w_2^{\prime}, \mu_2^{\prime}\right)$, where $\lambda_0(M)$ is defined by \eqref{2-7},  $r_2^{\prime}$ and $\sigma_2^{\prime}$ are as in Corollary \ref{cor3} with $w_2^{\prime}=\cosh \sigma_2^{\prime}$ and $\mu_2^{\prime}=\sinh \sigma_2^{\prime}$.
%

In particular, if $|\log f(z)|\leq 1$ for all $z\in\mathbb{D}$ and $f(0)=\left|J_f(0)\right|=1$, then $f$ is univalent in the unit disk $\mathbb{D}$ and $f\left(\mathbb{D}\right)$ contains a schlicht disk $\mathbb{D}\left(w_1', \mu_1'\right)$, where $w_1'=\cosh 1$ and $\mu_1'=\sinh 1$. This result is sharp.
\end{corollary}

\begin{corollary}\label{cor8a} Suppose that $M\geq 1$, and $f(z)=\left(a_0(z)\right) \cdot\left(\bar{z} a_1(z)\right) \cdot\left(\bar{z}^2 a_2(z)\right)$ is a $\log$-$3$-analytic function on $\mathbb{D}$, satisfying $f(0)=\left|J_f(0)\right|=1$, where for each $k$, $a_k(z)$ is $\log$-analytic on $\mathbb{D}$, and $A_k(z):=\log a_k(z)=\sum_{n=0}^{\infty} a_{n, k} z^n$, and all its non-zero coefficients $a_{n_1,k_1}, a_{n,k_2}$ satisfy $\left|\arg \frac{a_{n_1, k_1}}{a_{n_2, k_2}}\right| \leq \frac{\pi}{2}$ for each $k_1,k_2\in\{0, 1 , 2\}$ and $n_1, n_2\in\mathbb{N}$ with $k_1\neq k_2,\, n_1\neq n_2$. If $|\log f(z)| \leq M$, then $f$ is univalent in the disk $\mathbb{D}_{r_2^*}$ and $f\left(\mathbb{D}_{r_2^*}\right)$ contains a schlicht disk  $\mathbb{D}\left(w_2^*, \mu_2^*\right)$, where $\lambda_0(M)$ is defined by \eqref{2-7},  $r_2^*$ and  $\sigma_2^*$ are as in Corollary \ref{cor4} with $w_2^*=\cosh \sigma_2^*$ and $\mu_2^*=\sinh \sigma_2^*$.
and $w_2^*=\cosh \sigma_2^*$, $\mu_2^*=\sinh \sigma_2^*$.

In particular, if $|\log f(z)|\leq 1$ for all $z\in\mathbb{D}$ and $f(0)=\left|J_f(0)\right|=1$, then $f$ is univalent in the unit disk $\mathbb{D}$ and $f\left(\mathbb{D}\right)$ contains a schlicht disk $\mathbb{D}\left(w_1', \mu_1'\right)$, where $w_1'=\cosh 1$ and $\mu_1'=\sinh 1$. This result is sharp.
\end{corollary}

\begin{theorem}\label{thm8} Suppose that $M\geq 1$, and $f(z)=\prod_{k=0}^{\alpha-1}\left(a_k(z)\right)^{\bar{z}^k}$ is a $\log$-$\alpha$-analytic function on $\mathbb{D}$, satisfying $f(0)=\lambda_f(0)=1$, where for each $k$ , $a_k(z)$ is $\log$-analytic on $\mathbb{D}$, and $A_k(z):=\log a_k(z)=\sum_{n=0}^{\infty} a_{n, k} z^n$, and all its non-zero coefficients $a_{n_1, k_1}, a_{n_2, k_2}$ satisfy $\left|\arg \frac{a_{n_1, k_1}}{a_{n_2, k_2}}\right| \leq \frac{\pi}{2}$ for each $k_1,k_2\in\{0,\ldots , \alpha-1\}$ and $n_1, n_2\in\mathbb{N}$ with $k_1\neq k_2,\, n_1\neq n_2$. If $|\log f(z)| \leq M$, then $f$ is univalent in the disk $\mathbb{D}_{r_3}$ and $f\left(\mathbb{D}_{r_3}\right)$ contains a schlicht disk $\mathbb{D}\left(w_3, \mu_3\right)$, where $r_3$ and $\sigma_3$ are as in Theorem \ref{thm5} with
$w_3=\cosh \sigma_3$ and $\mu_3=\sinh \sigma_3$.

In particular, if $|\log f(z)|\leq 1$ for all $z\in\mathbb{D}$ and $f(0)=\lambda_f(0)=1$, then $f$ is univalent in $\mathbb{D}$ and $f\left(\mathbb{D}\right)$ contains a schlicht disk $\mathbb{D}\left(w_1', \mu_1'\right)$, where $w_1'=\cosh 1\text{ and } \mu_1'=\sinh 1$. This result is sharp.
\end{theorem}


\begin{corollary}\label{cor7} Suppose that $M\geq 1$, and $f(z)=\left(a_0(z)\right) \cdot\left(\bar{z} a_1(z)\right)$ is a $\log$-$2$-analytic function on $\mathbb{D}$, satisfying $f(0)=\lambda_f(0)=1$, where $a_0(z), a_1(z)$ are $\log$-analytic on $\mathbb{D}$, and $A_k(z):=\log a_k(z)=\sum_{n=0}^{\infty} a_{n, k} z^n$ for $k=0,1$, and all its non-zero coefficients  $a_{n_1,0}, a_{n,1}$ satisfy $\left|\arg \frac{a_{n_1, 0}}{a_{n_2, 1}}\right| \leq \frac{\pi}{2}$ for each $n_1, n_2\in\mathbb{N}$ with $n_1\neq n_2$. If $|\log f(z)| \leq M$, then $f$ is univalent in the disk $\mathbb{D}_{r_3^{\prime}}$ and $f\left(\mathbb{D}_{r_3^{\prime}}\right)$ contains a schlicht disk  $\mathbb{D}\left(w_3^{\prime}, \mu_3^{\prime}\right)$, where
$$r_3^{\prime}=1-\sqrt{\frac{2\sqrt{M^2-1}}{1+2\sqrt{M^2-1}}},
$$
$w_3^{\prime}=\cosh \sigma_3^{\prime}$, $\mu_3^{\prime}=\sinh \sigma_3^{\prime}$, and
$$\sigma_3^{\prime}=r_3^{\prime}-\sqrt{M^2-1} \frac{2 {r_3^{\prime}}^2}{1-r_3^{\prime}}.$$

In particular, if $|\log f(z)|\leq 1$ for all $z\in\mathbb{D}$ and $f(0)=\left|J_f(0)\right|=1$, then $f$ is univalent in the unit disk $\mathbb{D}$ and $f\left(\mathbb{D}\right)$ contains a schlicht disk $\mathbb{D}\left(w_1', \mu_1'\right)$, where $w_1'=\cosh 1,\, \mu_1'=\sinh 1$. This result is sharp.
\end{corollary}

\begin{corollary}\label{cor8} Suppose that $M\geq 1$, and $f(z)=\left(a_0(z)\right) \cdot\left(\bar{z} a_1(z)\right) \cdot\left(\bar{z}^2 a_2(z)\right)$ is a $\log$-$3$-analytic function on $\mathbb{D}$, satisfying $f(0)=\lambda_f(0)=1$, where for each $k$, $a_k(z)$ is $\log$-analytic on $\mathbb{D}$, and $A_k(z):=\log a_k(z)=\sum_{n=0}^{\infty} a_{n, k} z^n$, and all its non-zero coefficients $a_{n_1,k_1}, a_{n,k_2}$ satisfy $\left|\arg \frac{a_{n_1, k_1}}{a_{n_2, k_2}}\right| \leq \frac{\pi}{2}$ for each $k_1,k_2\in\{0, 1 , 2\}$ and $n_1, n_2\in\mathbb{N}$ with $k_1\neq k_2,\, n_1\neq n_2$. If $|\log f(z)| \leq M$, then $f$ is univalent in the disk $\mathbb{D}_{r_3^*}$ and $f\left(\mathbb{D}_{r_3^*}\right)$ contains a schlicht disk  $\mathbb{D}\left(w_3^*, \mu_3^*\right)$, where $$r_3^*=1-\sqrt{\frac{3\sqrt{M^2-1}}{1+3\sqrt{M^2-1}}},$$
$w_3^*=\cosh \sigma_3^*$, $\mu_3^*=\sinh \sigma_3^*$, and
$$\sigma_3^*= r_3^*-\sqrt{M^2-1} \frac{3{r_3^*}^2}{1-r_3^*}.$$

In particular, if $|\log f(z)|\leq 1$ for all $z\in\mathbb{D}$ and $f(0)=\left|J_f(0)\right|=1$, then $f$ is univalent in the unit disk $\mathbb{D}$ and $f\left(\mathbb{D}\right)$ contains a schlicht disk $\mathbb{D}\left(w_1', \mu_1'\right)$, where $w_1'=\cosh 1,\, \mu_1'=\sinh 1$. This result is sharp.
\end{corollary}

\subsection*{Acknowledgments}
The work of the first two authors was supported by Natural Science Foundation of Guangdong Province (Grant No. 2021A1515010058).
The authors thank the referee  for his/her valuable comments and suggestions to this paper.

\subsection*{Conflict of Interests}
The authors declare that they have no conflict of interest, regarding the publication of this paper.
\subsection*{Data Availability Statement}
The authors declare that this research is purely theoretical and does not associate with any data.

\end{document}